\documentclass[12pt,a4paper]{article}
\usepackage{amsmath,amssymb,amsbsy,amscd,amsthm}
\newtheorem{thm}{Theorem}[section]
\newtheorem{lem}[thm]{Lemma}

\newcommand{\ch }{\mathop{\rm char}\nolimits}
\newcommand{\Spec }{\mathop{\rm Spec}\nolimits}

\newcommand{\id }{{\rm id}}

\newcommand{\degw }{\mathop{\deg _{\w }}\nolimits}
\newcommand{\Aut }{\mathop{\rm Aut}\nolimits}

\newcommand{\zs}{\{ 0\} }
\newcommand{\sm}{\setminus}
\newcommand{\G}{{\bf G}}
\newcommand{\Ga}{{\bf G}_a}

\newcommand{\R}{{\bf R}}

\newcommand{\Z}{{\bf Z}}

\newcommand{\w}{{\bf w}}

\begin{document}

\title{The automorphism theorem and 
additive group actions on the affine plane}

\author{Shigeru Kuroda
\thanks{Partly supported by JSPS KAKENHI 
Grant Number 15K04826.}}

\date{}


\maketitle

\begin{abstract}
Due to Rentschler, Miyanishi and Kojima, 
the invariant ring for a 
$\Ga $-action on the affine plane over an arbitrary field 
is generated by one coordinate. 
In this note, 
we give a new short proof for this result 
using the automorphism theorem of Jung and van der Kulk. 
\end{abstract}

\section{Introduction}
\setcounter{equation}{0}

Let $k$ be a field, 
$A$ a $k$-domain, 
$A[T]$ the polynomial ring in one variable over $A$, 
and $\sigma :A\to A[T]$ a homomorphism of $k$-algebras. 
Then, 
$\sigma $ defines an action of the additive group 
$\Ga =\Spec k[T]$ on $\Spec A$ 
if and only if the following holds for each $a\in A$, 
where we write 
$\sigma (a)=\sum _{i\ge 0}a_iT^i$ with $a_i\in A$, 
and $U$ is a new variable: 

\smallskip 

\quad 
(A1) $a_0=a$. 
\qquad 
(A2) $\sum _{i\ge 0}\sigma (a_i)U^i=\sum _{i\ge 0}a_i(T+U)^i$ 
in $A[T,U]$. 

\smallskip 

\noindent 
If this is the case, 
we call $\sigma $ a $\Ga $-{\it action} on $A$. 
The $\sigma $-{\it invariant ring} 
$A^{\sigma }:=\{ a\in A\mid \sigma (a)=a\} $ 
is equal to $\sigma ^{-1}(A)$ by (A1). 
We say that $\sigma $ is {\it nontrivial} 
if $A^{\sigma }\ne A$.

Now, 
let $k[x_1,x_2]$ 
be the polynomial ring in two variables over $k$, 
and $\Aut _kk[x_1,x_2]$ the automorphism group 
of the $k$-algebra $k[x_1,x_2]$. 
We often express $\phi \in \Aut _kk[x_1,x_2]$ 
as $(\phi (x_1),\phi (x_2))$. 
We call $f\in k[x_1,x_2]$ a {\it coordinate} of 
$k[x_1,x_2]$ if there exists $g\in k[x_1,x_2]$ 
such that $(f,g)$ belongs to $\Aut _kk[x_1,x_2]$, 
that is, 
$k[f,g]=k[x_1,x_2]$.

The following theorem is a fundamental result 
for $\Ga $-actions on $k[x_1,x_2]$.

\begin{thm}\label{thm:Ga}
For every nontrivial $\G _a$-action 
$\sigma $ on $k[x_1,x_2]$, 
there exists a coordinate $f$ of $k[x_1,x_2]$ 
such that $k[x_1,x_2]^{\sigma }=k[f]$. 
\end{thm}

This theorem was first proved by Rentschler~\cite{Rentschler} 
when $\ch k=0$ in 1968, 
and then by Miyanishi~\cite{Miyanishi nagoya} 
when $k$ is algebraically closed in 1971. 
Recently, 
Kojima~\cite{Kojima} proved the general case 
by making use of Russell-Sathaye~\cite{RS} 
(see also \cite{Kuroda}).

For each $f\in k[x_1,x_2]$, 
we denote by $\deg f$ the total degree of $f$, 
and by $\bar{f}$ or $(f)^-$ 
the highest homogeneous part of $f$ 
for the standard grading on $k[x_1,x_2]$. 
The following well-known theorem was first proved by Jung~\cite{Jung} 
when $\ch k=0$ in 1942. 
The general case was proved 
by van der Kulk~\cite{Kulk} in 1953 
(see also the proof of Makar-Limanov~\cite{ML} 
and its modifications by Dicks~\cite{Dicks} 
and Cohn~\cite[Thm.\ 8.5]{Cohn}).

\begin{thm}
\label{thm:autom}
For every $(f_1,f_2)\in \Aut _kk[x_1,x_2]$ with 
$\deg f_1\ge 2$ or $\deg f_2\ge 2$, 
there exist $(i,j)\in \{ (1,2),(2,1)\} $, 
$\alpha \in k^*$ and $l\ge 1$ such that 
$\bar{f_i}=\alpha \bar{f_j}^l$
\end{thm}

The purpose of this note is to give a new short proof 
of Theorem~\ref{thm:Ga} based on Theorem~\ref{thm:autom} 
(cf.~\S 2). 
We should mention that, 
if $k$ is an infinite field, 
Theorem~\ref{thm:Ga} can be derived from Theorem~\ref{thm:autom} 
by a group-theoretic approach (cf.~\cite{Kam}). 
Our approach is different from this approach, 
and is valid for an arbitrary $k$.

Conversely, 
Theorem~\ref{thm:autom} 
can be derived easily from Theorem~\ref{thm:Ga}. 
This seems known to experts, 
at least when $\ch k=0$ 
(cf.~e.g.~\cite[\S 5.1]{Essen} for related discussion). 
For completeness, 
we also give a proof for this implication (cf.~\S 3).

The author thanks Prof.\ Ryuji Tanimoto 
(Shizuoka University) for discussion.

\section{$\Ga $-action}
\setcounter{equation}{0}

Recall that, 
if $\sigma $ is a nontrivial $\Ga $-action on $A$, 
then $A$ has transcendence degree one over $A^{\sigma }$ 
(cf.~\cite[\S 1.5]{Miyanishi Tata}). 
For each $t\in A^{\sigma }$, 
the map 
$\sigma _t:A\stackrel{\sigma }{\to }A[T]\ni 
f(T)\mapsto f(t)\in A$ 
is an automorphism of the $k$-algebra $A$. 
Actually, 
we have $\sigma _0=\id _A$ by (A1), 
and $\sigma _t\circ \sigma _u=\sigma _{t+u}$ 
for each $t,u\in A^{\sigma }$ by (A2).

Now, 
we derive Theorem~\ref{thm:Ga} 
from Theorem~\ref{thm:autom}. 
For each 
$q=\sum _{i\ge 0}q_iT^i\in k[x_1,x_2][T]\sm \zs $ 
with $q_i\in k[x_1,x_2]$, 
we define the $\Z ^2$-{\it degree} of $q$ by 
$$
\deg _{\Z ^2}q:=\max \{ (i,\deg q_i)\in \Z ^2\mid 
i\ge 0,\ q_i\ne 0\} , 
$$
where $\Z ^2$ is ordered lexicographically, 
i.e., $(a,b)\le (a',b')$ if and only if $a<a'$, 
or $a=a'$ and $b\le b'$. 
Let $\sigma $ be any nontrivial $\G _a$-action on $k[x_1,x_2]$. 
It suffices to find $(f_1,f_2)\in \Aut _kk[x_1,x_2]$ 
for which $\sigma (f_1)$ or $\sigma (f_2)$ 
belongs to $k[x_1,x_2]$. 
Suppose that such an $(f_1,f_2)$ does not exist. 
Choose $(f_1,f_2)\in \Aut _kk[x_1,x_2]$ so that 
$\deg _{\Z ^2}\sigma (f_1)+\deg _{\Z ^2}\sigma (f_2)$ 
is minimal, 
and write $\sigma (f_i)=q_i(T)=\sum _{j=1}^{m_i}q_{i,j}T^j$ 
for $i=1,2$, 
where $q_{i,j}\in k[x_1,x_2]$ with $q_{i,m_i}\ne 0$. 
By supposition, 
we have $m_1,m_2\ge 1$. 
Since $k[x_1,x_2]^{\sigma }\ne k$, 
we may take $g\in k[x_1,x_2]^{\sigma }$ 
with $g\not\in k$. 
Then, 
$\sigma _{g^r}=(q_1(g^r),q_2(g^r))$ 
belongs to $\Aut _kk[x_1,x_2]$ 
for each $r\ge 0$ as mentioned above. 
Since $\deg g\ge 1$, 
there exists $r_0>0$ such that, 
for each $r\ge r_0$ and $i=1,2$, 
we have 
$(q_i(g^r))^-=\bar{q}_{i,m_i}\bar{g}^{rm_i}$ 
and $\deg q_i(g^r)\ge 2$. 
By Theorem~\ref{thm:autom}, 
for each $r\ge r_0$, 
there exist $(i,j)\in \{ (1,2),(2,1)\} $, 
$\alpha \in k^*$ and $l\ge 1$ such that 
$\bar{q}_{i,m_i}\bar{g}^{rm_i}
=\alpha (\bar{q}_{j,m_j}\bar{g}^{rm_j})^l$. 
This equality implies that 
\begin{equation}\label{eq:initial eq}
r(m_i-lm_j)\deg g=l\deg q_{j,m_j}-\deg q_{i,m_i}. 
\end{equation}
We note that $(i,j)$, $\alpha $ and $l$ above depend on $r$. 
By (\ref{eq:initial eq}), 
we see that $m_i=lm_j$ holds for sufficiently large $r$. 
Take such an $r$. 
Then, 
we have 
$\bar{q}_{i,m_i}=\alpha \bar{q}_{j,m_j}^l$, 
and hence $\deg (q_{i,m_i}-\alpha q_{j,m_j}^l)<\deg q_{i,m_i}$. 
Thus, the $\Z ^2$-degree of 
\begin{align*}
\sigma (f_i-\alpha f_j^l)
&=q_i(T)-\alpha q_j(T)^l \\
&=(q_{i,m_i}-\alpha q_{j,m_j}^l)T^{m_i}+
(\text{terms of lower degree in }T)
\end{align*}
is strictly less than that of $\sigma (f_i)$. 
Since $(f_i-\alpha f_j^l, f_j)$ 
belongs to $\Aut _kk[x_1,x_2]$, 
this contradicts the minimality of 
$\deg _{\Z ^2}\sigma (f_1)+\deg _{\Z ^2}\sigma (f_2)$, 
completing the proof.

\section{Automorphism Theorem}
\setcounter{equation}{0}

We derive Theorem~\ref{thm:autom} 
from Theorem~\ref{thm:Ga}. 
Let $f=\sum _{i_1,i_2\ge 0}u_{i_1,i_2}x_1^{i_1}x_2^{i_2}$ be 
an element of $k[x_1,x_2]\sm \zs $, 
where $u_{i_1,i_2}\in k$. 
For each $\w =(w_1,w_2)\in \R ^2$, 
we define 
$$
\degw f:=\max \{ i_1w_1+i_2w_2\mid 
i_1,i_2\ge 0,\ u_{i_1,i_2}\ne 0\} 
\text{ and }
f^{\w }:={\sum }'u_{i_1,i_2}x_1^{i_1}x_2^{i_2}, 
$$
where the sum $\sum '$ is taken over $i_1,i_2\ge 0$ 
with $i_1w_1+i_2w_2=\degw f$. 
We say that $f$ is $\w $-{\it homogeneous} if $f^{\w }=f$, 
and {\it non-univariate} if $f\not\in k[x_1]\cup k[x_2]$. 
We define  
$$
\w (f):=(\deg _{(0,1)}f,\deg _{(1,0)}f). 
$$
We remark that $f^{\w (f)}$ is non-univariate if 
$f$ is non-univariate.

The following lemma is a consequence of Theorem~\ref{thm:Ga}.

\begin{lem}\label{lem:IC}
If $\sigma $ is a nontrivial $\Ga $-action on $k[x_1,x_2]$, 
and $f\in k[x_1,x_2]^{\sigma }$ is non-univariate, 
then there exist $a,b\in k^*$, 
$(i,j)\in \{ (1,2),(2,1)\} $ 
and $l,m\ge 1$ such that 
\begin{equation}\label{eq:IC}
f^{\w (f)}=a(x_i-bx_j^l)^m. 
\end{equation}
\end{lem}
\begin{proof}
Since $f$ is non-univariate, 
so is $f^{\w (f)}$ as remarked. 
By Derksen--Hadas--Makar-Limanov~\cite[Prop.\ 2.2]{DHM}, 
there exists a nontrivial $\Ga $-action $\tau $ on $k[x_1,x_2]$ 
such that $f^{\w (f)}$ belongs to $k[x_1,x_2]^{\tau }$. 
By Theorem~\ref{thm:Ga}, 
$k[x_1,x_2]^{\tau }=k[h]$ 
holds for some coordinate $h$ of $k[x_1,x_2]$. 
We may assume that $h$ has no constant term. 
Then, 
since $f^{\w (f)}$ belongs to $k[h]\sm k$ 
and $f^{\w (f)}$ is $\w (f)$-homogeneous, 
we see that $h$ is $\w (f)$-homogeneous, 
and $f^{\w (f)}=\alpha h^m$ 
for some $\alpha \in k^*$ and $m\ge 1$. 
This implies that $h$ is non-univariate. 
Since $h$ is a $\w (f)$-homogeneous coordinate, 
$h$ must have the form $\beta x_i+\gamma x_j^l$ 
for some $\beta, \gamma \in k^*$ and $l\ge 1$. 
Therefore, 
$f^{\w (f)}$ is written as in (\ref{eq:IC}). 
\end{proof}

Now, 
we prove Theorem~\ref{thm:autom}. 
Take $\phi =(f_1,f_2)\in \Aut _kk[x_1,x_2]$. 
Set $w_i:=\deg f_i$ 
and $g_i:=\phi ^{-1}(x_i)$ for $i=1,2$. 
Assume that $w_1\ge 2$ or $w_2\ge 2$. 
Then, 
there exists $t\in \{ 1,2\} $ 
such that $g_t$ is non-univariate. 
Note that a nontrivial 
$\Ga $-action $\sigma $ on $k[x_1,x_2]$ 
is defined by 
$\sigma (g_t)=g_t$ and $\sigma (g_u)=g_u+T$, 
where $u\ne t$. 
Since $g_t$ belongs to $k[x_1,x_2]^{\sigma }$, 
we may write $g_t^{\w (g_t)}$ as in (\ref{eq:IC}) 
by Lemma~\ref{lem:IC}. 
Set $\w :=(w_1,w_2)$. 
Then, 
we have 
$$
\degw g_t=m\max \{ w_i,lw_j\} \ge \max \{ w_1,w_2\} \ge 2
>1=\deg x_t=\deg \phi (g_t). 
$$
This implies that $a(\bar{f}_i-b\bar{f}_j^l)^m=0$. 
Therefore, 
we get $\bar{f}_i=b\bar{f}_j^l$, 
proving Theorem~\ref{thm:autom}.

\noindent
Department of Mathematics and Information Sciences \\ 
Tokyo Metropolitan University \\ 
1-1  Minami-Osawa, Hachioji, 
Tokyo 192-0397, Japan\\
kuroda@tmu.ac.jp

\end{document}